\documentclass{amsart}
\usepackage{amsfonts}
\usepackage{amsmath}
\usepackage{amssymb}
\usepackage{graphicx}
\usepackage{verbatim}
\usepackage{color}
\newtheorem{thm}{Theorem}
\theoremstyle{plain}

\newtheorem{conjecture}{Conjecture}
\newtheorem{cor}{Corollary}

\newtheorem{lem}{Lemma}

\newtheorem{problem}{Problem}

\newtheorem{prop}{Proposition}
\newtheorem{rem}{Remark}

\numberwithin{equation}{section}

\newcommand{\seq}[1]{\{#1\}_{k=0}^{\infty}}
\newcommand{\R}{\mathbb{R}}
\newcommand{\N}{\mathbb{N}}
\newcommand{\C}{\mathbb{C}}

\title{Polynomially Interpolated Legendre Multiplier Sequences}

\author{Matthew Chasse, Tam\'as Forg\'acs, and Andrzej Piotrowski}

\begin{document}
\begin{abstract}
We prove that every multiplier sequence for the Legendre basis which can be interpolated by a polynomial has the form $\seq{h(k^2+k)}$, where $h\in\mathbb{R}[x]$. We also prove that a non-trivial collection of polynomials of a certain form interpolate multiplier sequences for the Legendre basis, and we state conjectures on how to extend these results. 
\end{abstract}

\maketitle

\section{Introduction}\label{s:Introduction}
Over the past decade, there has been an effort to characterize multiplier sequences acting on various orthogonal bases for $\mathbb{R}[x]$ (see \cite{bates}, \cite{bdfu}, \cite{bc}, \cite{bo}, \cite{nreup}, \cite{fhms}, \cite{FP}, \cite{P}). The present work focuses on the Legendre polynomial basis $\{P_k\}_{k=0}^{\infty}$. Recall that, for each nonnegative integer $k$, the $k$-th Legendre polynomial $P_k$ can be defined as the polynomial solution to Legendre's differential equation  
$$
(x^2-1)y'' + 2 x y' - k(k+1) y= 0,
$$
normalized so that $P_k(1)=1$. As such, our work is a continuation of the investigations carried out in \cite{bdfu} and \cite{fhms}. 

We now recall some of the relevant terminology in this subject area that will be used freely throughout the paper. A polynomial is called {\it hyperbolic} if all of its zeros are real. An operator on $\mathbb{R}[x]$ that maps the class of hyperbolic polynomials into itself is called a {\it hyperbolicity preserver}. If $T$ is a linear operator on $\mathbb{R}[x]$ which is diagonal with respect to the Legendre basis, and if $T$ is also a hyperbolicity preserver, then the corresponding eigenvalue sequence $\{\gamma_k\}_{k=0}^{\infty}$ for which $T[P_k(x)] = \gamma_k P_k(x)$ is called a {\it multiplier sequence for the Legendre basis}. 

Although the bulk of the present work has to do with hyperbolicity preserving operators, there will also be a brief discussion (in Section \ref{s:collection}) of complex zero decreasing operators and sequences which are defined as follows. If an operator $T$ on $\mathbb{R}[x]$ has the property that $Z_c(T[p])\leq Z_c(p)$ for all polynomials $p$, where $Z_c(f)$ denotes the number of nonreal zeros of $f$, then $T$ is called a {\it complex zero decreasing operator}.  If $T$ is a linear operator on $\mathbb{R}[x]$ which is diagonal with respect to the Legendre basis and $T$ is a complex zero decreasing operator, then the corresponding eigenvalue sequence is called a {\it complex zero decreasing sequence for the Legendre basis}. It is straightforward to see that every complex zero decreasing operator is a hyperbolicity preserving operator. There are, however, hyperbolicity preserving operators which are not complex zero decreasing operators. Finally, we are obliged to note that the function which is identically zero plays a peculiar role in this theory. As usual, we will declare that this function has only real zeros and define $Z_c(0)=0$. 

Central to the theory of multiplier sequences is the {\it Laguerre-P\'olya Class} of functions, denoted $\mathcal{L-P}$, which can be defined as follows. A function belongs to the class $\mathcal{L-P}$ if and only if it is a uniform limit, on compact subsets of $\mathbb{C}$, of polynomials with only real zeros. Equivalently, a real entire function $\varphi\in \mathcal{L-P}$, if and only if it can be expressed in the form
$$\varphi(x) = cx^me^{-ax^2 + bx}\prod_{k=1}^\omega\left(1+\frac{x}{x_k}\right)e^{\frac{-x}{x_k}} \text{\hspace{5 mm}} (0\le\omega\le\infty),$$
where $b,c,x_k \in \mathbb{R}$, $m$ is a non-negative integer, $a\ge 0$, $x_k\neq 0$, and 
$\sum_{k=1}^\omega\frac{1}{x_k^2}<\infty$.

\begin{prop}\cite[p. 242]{CCturan}\label{TI}
Let $\varphi(x)=\sum_{k=0}^\infty \frac{\gamma_k}{k!}x^k\in\mathcal{L-P}$.  Then the Tur\'{a}n inequalities hold for the Taylor coefficients of $\varphi$; that is,
\begin{equation}
\gamma_k^2-\gamma_{k-1}\gamma_{k+1} \ge 0 \text{, \hspace{5mm} } k=1, 2, 3, \ldots.
\end{equation}
\end{prop}  

We refer the reader to \cite{CCsurvey} for a comprehensive treatment of the Laguerre-P\'olya Class. For our purposes, we shall only need the fact that an even function in $\mathcal{L-P}$ must have coefficients (of even powers of the variable) which alternate in sign; this follows immediately from Proposition \ref{TI}. 

In this paper, we focus mainly on sequences that can be interpolated by polynomials. That is to say, sequences of the form $\{p(k)\}_{k=0}^{\infty}$ where $p\in\mathbb{R}[x]$. In Theorem \ref{x^2+x} of Section \ref{s:Form and Order} we prove that any multiplier sequence for the Legendre basis which can be interpolated by a polynomial must have the form $\seq{h(k^2+k)}$, where $h\in\mathbb{R}[x]$. In Section \ref{s:collection}, we prove that a certain type of differential operator is hyperbolicity preserving (Theorem \ref{fall}) and use that result to establish the existence of a new and non-trivial collection of multiplier sequences for the Legendre basis (Corollary \ref{fallseq}). The paper concludes with two conjectures concerning Legendre multiplier sequences interpolated by polynomials of even degree.

\section{Preliminaries}\label{s:Preliminaries}
It is known that any linear operator $T$ on $\mathbb{R}[x]$ can be represented as a differential operator with coefficients $S_k\in\mathbb{R}[x]$, i.e., 
$$
T = \sum_{k=0}^{\infty} S_k(x) D^k.
$$
The goal of this section is to develop a formula for $S_k(0)$ in the case where $T$ is diagonal with respect to the Legendre basis and has an eigenvalue sequence that can be interpolated by a polynomial. This formula will be useful in proving Theorem \ref{thm:noodd} in the next section, which states that a polynomially interpolated multiplier sequence for the Legendre basis must be interpolated by a polynomial of even degree. This formula will be useful in proving Theorem \ref{thm:noodd} in the next section, which states that if a multiplier sequence for the Legendre basis is interpolated by a polynomial $p$, then $p$ must be of even degree.

\begin{lem} \label{Q0lem}
For any differential operator
$$
T = \sum_{k=0}^{\infty} Q_k(x) D^k
$$
on $\R[x]$, the coefficient polynomials evaluated at 0 can be computed by 
$$
Q_n(0) =  \frac{1}{n!} \left[T[x^n]\right]_{x=0} \qquad (n=0, 1, 2, \dots).
$$
\end{lem}

\begin{proof}
The result follows from calculating 
$$
T[x^n] = Q_0(x) x^n + n x^{n-1} Q_1(x) + \cdots + n! Q_n(x),
$$
and evaluating this expression at $x=0$.
\end{proof}

We now specialize the previous lemma to operators which are diagonal with respect to the Legendre polynomial basis. Throughout the rest of this paper we use the Pocchammer symbol to denote the rising factorial
$$
(\alpha)_k = \alpha(\alpha+1)(\alpha+2)\cdots (\alpha+k-1), \qquad (\alpha\in\R, k\in\N)
$$
and follow the convention that $(\alpha)_0=1$.
\begin{lem}\label{cor1}
Let $T$ be a linear operator on $\mathbb{R}[x]$ that satisfies $T[P_n]=\gamma_n P_n$, where $\{\gamma_n\}_{n=0}^{\infty}$ is a sequence of real numbers, and let the differential operator representation of $T$ be given by 
$$
T = \sum_{k=0}^{\infty} S_k(x) D^k.
$$ 
Then
\begin{equation}\label{even}
S_{2m+1}(0)=0 \qquad (m=0, 1, 2, \dots)
\end{equation}
and
\begin{equation}\label{s2m1}
S_{2m}(0) =  \frac{1}{2\cdot 4^m m!} \sum_{k=0}^{m} \binom{m}{k}\frac{(4k + 1) \gamma_{2k} (-1)^{k}}{(k+1/2)_{m+1}}\qquad (m=0, 1, 2, \dots).
\end{equation}
\end{lem}

\begin{proof}
We apply Lemma \ref{Q0lem} to the expansion of $x^n$ in terms of Legendre polynomials (see \cite[p. 181]{R}), 
$$
x^n = \frac{n!}{2^n} \sum_{k=0}^{[n/2]} \frac{(2n - 4k + 1) P_{n-2k}(x)}{k! (3/2)_{n-k}}
$$
to obtain 
$$
S_n(0) = \frac{1}{2^n} \sum_{k=0}^{[n/2]} \frac{(2n - 4k + 1) \gamma_{n-2k} P_{n-2k}(0)}{k! (3/2)_{n-k}}.
$$
The relation (see \cite[p. 158]{R})
$$
P_{2k+1}(0) = 0 \qquad (k=0, 1, 2, \dots)
$$
shows that $S_{2m+1}(0) = 0$ for all $m$. The relation (see \cite[p. 158]{R})
$$
P_{2k}(0) = \frac{(-1)^k(1/2)_k}{k!} \qquad (k=0, 1, 2, \dots)
$$ 
yields
\begin{align*}
S_{2m}(0) &=  \frac{1}{4^m} \sum_{k=0}^{m} \frac{(4m - 4k + 1) \gamma_{2m-2k} (-1)^{m-k}(1/2)_{m-k}}{k! (3/2)_{2m-k} (m-k)!}\\
 &=\frac{1}{4^m m!} \sum_{k=0}^{m} \binom{m}{k} (4k + 1) \gamma_{2k} (-1)^{k} \frac{(1/2)_{k}}{(3/2)_{m+k}}
\end{align*}
which can be re-written in the form (\ref{s2m1}) as desired.
\end{proof}

The fact that the odd indexed coefficient polynomials vanish at zero (equation (\ref{even})) will be useful to us in the next section. Continuing to specialize, we now focus on the even indexed coefficient polynomials in the case where the eigenvalue sequence $\{\gamma_k\}_{k=0}^{\infty}$ can be interpolated by a polynomial $p$. In this case, we note that the quantity $(4k + 1) \gamma_{2k}= (4k+1) p(2k)$ that appears in equation (\ref{s2m1}) is a polynomial in $k$. This motivates us to find a closed form for the sum
\begin{equation}\label{sigma}
\sigma_{m,n} = \sum_{k=0}^{m} \binom{m}{k} \frac{k^n(-1)^k}{(k+1/2)_{m+1}}. 
\end{equation}

\begin{lem}\label{sigmalem1}
For $m,n=0, 1, 2, \dots$, the quantity $\sigma_{m,n}$ from equation (\ref{sigma}) satisfies
$$
\sigma_{m,n} = \left[ \theta^n  \frac{2}{(3/2)_m}{}_2F_1(-m, 1/2; m+3/2; x) \right]_{x=1} \qquad \left(\theta := x \frac{d}{dx}\right).
$$
\end{lem}

\begin{proof}
First note the relation
\begin{equation}\label{e1}
\frac{(-m)_k}{k!} = 
\begin{cases}
\binom{m}{k} (-1)^k  & \qquad k=0, 1, 2, \dots, m,\\
0 & \qquad k=m+1, m+2, \dots\\
\end{cases}
\end{equation}
which follows directly from equation (3) of \cite[p. 58]{R} and the definition of the rising factorial. Furthermore, from (see \cite[p. 23, equation (7)]{R})
\begin{equation}\label{gampoc}
(\alpha)_n = \frac{\Gamma(\alpha+n)}{\Gamma(\alpha)} \qquad (\alpha \notin \{0, -1, -2, \dots \})
\end{equation}
we have, for $k\in\{0, 1, 2, \dots m\}$,
\begin{equation}\label{e2}
\frac{2\cdot (1/2)_k}{(m+3/2)_k (3/2)_m} = \frac{\Gamma(k+1/2)}{\Gamma(m+k+3/2)} = \frac{1}{(k+1/2)_{m+1}}.
\end{equation}
Combining equations (\ref{e1}) and (\ref{e2}), we see that
\begin{align*}
\sum_{k=0}^{m} \binom{m}{k} \frac{(-1)^k x^k}{(k+1/2)_{m+1}} &= \frac{2}{(3/2)_m} \sum_{k=0}^{m} \frac{(-m)_k (1/2)_k x^k}{ k!(m+3/2)_k }\\
 &= \frac{2}{(3/2)_m}F(-m, 1/2; m+3/2; x),
\end{align*}
where $F(a,b;c;x)$ is the hypergeometric function (see \cite[p. 45]{R}). The conclusion that
$$
\sigma_{m,n} = \left[ \theta^n  \frac{2}{(3/2)_m}{}_2F_1(-m, 1/2; m+3/2; x) \right]_{x=1} \qquad \left(\theta := x \frac{d}{dx}\right)
$$
now follows from the fact that $\theta x^k = k x^k$, which holds for all $k$.
\end{proof}

The next lemma will be useful in obtaining another explicit formula for $\sigma_{m,n}$.

\begin{lem}\label{Q_k}
The coefficient polynomials $Q_k(x)$ in the expansion
$$
(xD)^n = \sum_{k=0}^{n} Q_k(x) D^k
$$
are given by $Q_k(x) = S(n,k) x^k,$ where 
\begin{equation}\label{c_k}
S(n,k) = \frac{1}{k!}\sum_{j=0}^k \binom{k}{j} (-1)^{k-j}j^n
\end{equation}
are the Stirling numbers of the second kind.
\end{lem}

\begin{proof}
This result is known, but we offer a simple proof here based on a formula for $Q_k$ for a linear operator $T:\C[x]\to\C[x]$ \cite[p. 106, Prop. 216]{C}:
\begin{equation*}
  Q_k(x) = \frac{1}{k!}\sum_{j=0}^k \binom{k}{j}T[x^j](-x)^{k-j}
\end{equation*}
substituting $T=(xD)^n$ yields
\begin{equation*}
  Q_k(x) = \frac{x^k}{k!}\sum_{j=0}^k \binom{k}{j} (-1)^{k-j}j^n, 
\end{equation*}
and the result follows.
\end{proof}

We now use the previous lemma along with some properties of the hypergeometric function to obtain another formula for $\sigma_{m,n}$.

\begin{lem}\label{lemsigma2}
For nonnegative integers $m$ and $n$ which satisfy $2m\ge n-1$, the quantity $\sigma_{m,n}$ from equation (\ref{sigma}) can be written in the form
$$
\sigma_{m,n} = \frac{2\cdot 4^m (2m-n)!}{m!(4m+1)!!} p_n(m), 
$$
where 
\begin{equation}\label{p_n}
p_n(m)=\sum_{k=0}^n S(n,k)(-m)_k(1/2)_k(2m-n+1)_{n-k}
\end{equation}
is polynomial of degree $n$ in the variable $m$. Here, $S(n,k)$ is the Stirling number of the second kind as defined in equation (\ref{c_k}).
\end{lem}

\begin{proof}
We first use Lemma \ref{sigmalem1} and Lemma \ref{Q_k} together with the relations (\cite[p. 49, l. 7]{R})
$$
F(a,b;c;1) = \frac{\Gamma(c) \Gamma(c-a-b)}{\Gamma(c-a) \Gamma(c-b)} \qquad (\text{Re}(c-a-b)>0)
$$
and (\cite[p. 69, ex. 1]{R})
$$
\frac{d}{dx} F(a, b;c;x) = \frac{ab}{c} F(a+1,b+1;c+1,x)
$$
to obtain
\begin{align*}
\sigma_{m,n} &= \frac{2}{(3/2)_m} \left[\sum_{k=0}^n S(n,k)x^k D^kF(-m,1/2;m+3/2,x)\right]_{x=1}\\
 &= \frac{2}{(3/2)_m} \left[\sum_{k=0}^n S(n,k) {x^k} \frac{(-m)_k(1/2)_k}{(m+3/2)_k}F\left(-m+k,\frac{1}{2}+k;m+\frac{3}{2}+k,x\right)\right]_{x=1}\\
 &= \frac{2}{(3/2)_m} \sum_{k=0}^n S(n,k) \frac{(-m)_k(1/2)_k}{(m+3/2)_k} \frac{\Gamma(m+3/2+k)\Gamma(2m-k+1)}{\Gamma(2m+3/2)\Gamma(m+1)},
\end{align*}
which is valid for $2m>n-1$. 

Next, we use the property (\ref{gampoc}) several times to obtain, for $2m\ge n-1$,
\begin{align*}
\sigma_{m,n} &= \frac{2}{m!}\frac{\Gamma(m+3/2)\Gamma(2m-n+1)}{(3/2)_m\Gamma(2m+3/2)} \sum_{k=0}^n S(n,k) (-m)_k(1/2)_k(2m-n+1)_{n-k}\\
&=\frac{2}{m!}\frac{\Gamma(m+3/2)\Gamma(2m-n+1)}{(3/2)_m\Gamma(2m+3/2)} p_n(m)\\
&=\frac{2}{m!}\frac{(2m-n)!}{(3/2)_m (m+3/2)_m} p_n(m)\\
&=\frac{2}{m!}\frac{(2m-n)!}{(3/2)_{2m} } p_n(m)\\
&=\frac{2\cdot 4^m (2m-n)!}{m!(4m+1)!!} p_n(m)
\end{align*}
where $p_n(m)$ is defined as in equation (\ref{p_n}). 
\end{proof}

With these results at hand, we can now give explicit expressions for $S_{2m}(0)$ in the case where the eigenvalue sequence $\{\gamma_k\}_{k=0}^{\infty}$ can be interpolated by the monomial $p(k)=k^n$.

\begin{lem}\label{kj}
If $\gamma_k = k^n$ then $S_{2m}(0)$ as defined in equation (\ref{s2m1}) is given by
\begin{equation}\label{s2m3}
S_{2m}(0) = \frac{2^n (2m-n-1)!}{(m!)^2 (4m+1)!!} \left[4 p_{n+1}(m)+ (2m-n)p_n(m)\right] \qquad (m\ge n),
\end{equation}
where $p_j(m)$ is defined as in equation (\ref{p_n}).
\end{lem}

\begin{proof}
Fix a nonnegative integer $n$ and let $\gamma_k=k^n$. Equations (\ref{s2m1}) and (\ref{sigma}) yield
\begin{align*}
S_{2m}(0) &= \frac{1}{2\cdot 4^m m! } \sum_{k=0}^{m} \binom{m}{k}\frac{(4k + 1) (2k)^n (-1)^{k}}{(k+1/2)_{m+1}}\\
&= \frac{2^{n-1}}{ 4^{m} m!}\left(4\cdot \sigma_{m,n+1} + \sigma_{m,n}\right).
\end{align*}
An application of Lemma \ref{lemsigma2}, some simplification, and comparison with equation (\ref{p_n}) then gives
$$
S_{2m}(0) = \frac{2^{n}}{(m!)^2(4m+1)!!}\left(4\cdot(2m-n-1)! p_{n+1}(m) + (2m-n)!p_n(m)\right),
$$
which is valid for $m\ge n$, and the result follows.
\end{proof}
We have now arrived at the goal of this section, which was to develop a useful formula for $S_{2m}(0)$ in the case where the eigenvalue sequence $\{\gamma_k\}_{k=0}^{\infty}$ can be interpolated by a polynomial $p(k) = \sum_{j=0}^{n} a_j k^j$. 
\begin{prop}
If $\gamma_k = \sum_{j=0}^{n} a_j k^j$ then, for all $m\ge n$, the quantity $S_{2m}(0)$ defined in equation (\ref{s2m1}) satisfies
\begin{equation}\label{s2m4}
S_{2m}(0) = \frac{(2m-n-1)!}{(m!)^2 (4m+1)!!} P(m),
\end{equation}
where 
\begin{equation}\label{P(m)}
P(m) = \sum_{j=0}^{n} a_j 2^j (2m-n)_{n-j} [4 p_{j+1}(m)+ (2m-j)p_j(m)]
\end{equation}
and $p_j(m)$ is defined as in equation (\ref{p_n}).
\end{prop}

\begin{proof}
Let $\gamma_k=\sum_{j=0}^{n} a_j k^j$. Equations (\ref{s2m1}) and (\ref{sigma}) yield
\begin{align*}
S_{2m}(0) &= \frac{1}{2\cdot 4^m m! } \sum_{k=0}^{m} \binom{m}{k}\frac{(4k + 1) \sum_{j=0}^{n}a_j(2k)^j (-1)^{k}}{(k+1/2)_{m+1}}\\
&= \sum_{j=0}^{n}a_j\left[\frac{1}{2\cdot 4^m m! }\sum_{k=0}^{m} \binom{m}{k}\frac{(4k + 1) (2k)^j (-1)^{k}}{(k+1/2)_{m+1}}\right].
\end{align*}
The quantity in the square brackets is the $S_{2m}(0)$ corresponding to $\gamma_k=k^j$, so Lemma \ref{kj} yields
\begin{align*}
S_{2m}(0) &=\sum_{j=0}^{n}a_j\frac{2^j (2m-j-1)!}{(m!)^2 (4m+1)!!} \left[4 p_{j+1}(m)+ (2m-j)p_j(m)\right],
\end{align*}
which is valid for $m\geq n$. Some manipulation of the factorials now gives the desired result.
\end{proof}

\section{Form and Order}\label{s:Form and Order}

In this section, we seek to determine what form a polynomial $p$ must have if the sequence $\{p(k)\}_{k=0}^{\infty}$ is a multiplier sequence for the Legendre basis. The conditions that we determine will, in turn, impose conditions on the order of the differential operator associated with the given sequence. We begin by showing that any such polynomial must have even degree. This settles open question (2) posed in section 5 of \cite{bdfu}, a paper of the second author.

\begin{thm} \label{thm:noodd}
Suppose that $p \in \mathbb{R}[x]$ is a polynomial of odd degree. Then $\seq{p(k)}$ is not a multiplier sequence for the Legendre basis.
\end{thm}

\begin{proof}
We argue by contradiction. Let $T$ be the operator corresponding to the sequence interpolated by the polynomial $p(x)=\sum_{j=0}^n a_jx^j$ with $n$ odd and $a_n \neq 0$, and suppose that $\{p(k)\}_{k=0}^{\infty}$ is a multiplier sequence for the Legendre basis. The symbol of $T$ is given by 
$$
G_T(x,y)= \sum_{k=0}^{\infty} \frac{(-1)^k T[x^k] y^k}{k!} = T[e^{-x y}], 
$$
where the operator $T$ acts on $e^{-xy}$ as a function of $x$ alone (see \cite{bb} for a comprehensive treatment of the symbol of an operator as it relates to hyperbolicity preserving operators). Suppose that the differential operator representation of $T$ is given by
$$
T = \sum_{k=0}^{\infty} S_k(x) D^k \qquad \left(D = \frac{d}{dx}\right).
$$
The symbol $G_T$ is then given by
$$
G_T(x,y) = e^{-xy} \sum_{k=0}^{\infty} S_k(x) (-1)^k y^k
$$
As noted in \cite{bo} and \cite{fhms}, we can act on this expression (as a function of $x$ alone) by the classical multiplier sequence $\{1, 0, 0, 0, \dots\}$ and the resulting function
$$
G_T(0,y) = \sum_{k=0}^{\infty} S_k(0) (-1)^k y^k
$$
must belong to the Laguerre-P\'olya class. By equation (\ref{even}), each term with odd index is zero. Thus, 
$$
G_T(0,y) = \sum_{k=0}^{\infty} S_{2m}(0) y^{2m}
$$
is an even function which belongs to the Laguerre-P\'olya class. It follows, by Proposition \ref{TI}, that the sequence $\{S_{2m}(0)\}_{m=0}^{\infty}$ must alternate in sign (a fact that we aim to contradict). By equations (\ref{s2m4}) and (\ref{P(m)}), for $m\ge n$, 
$$
S_{2m}(0) = \frac{(2m-n-1)!}{(m!)^2 (4m+1)!!} P(m),
$$
where 
$$
P(m) = \sum_{j=0}^{n} a_j 2^j (2m-n)_{n-j} [4 p_{j+1}(m)+ (2m-j)p_j(m)],
$$
and $p_j(m)$ is defined as in equation (\ref{p_n}). Note that $P(m)$ is a polynomial in $m$ which is either identically zero, or is not identically zero. In the latter case,  the sequence $\{P(m)\}_{m=0}^{\infty},$
and therefore the sequence $\{S_{2m}(0)\}_{m=0}^{\infty},$ eventually has constant sign, which would give us our desired contradiction. We will finish the proof by demonstrating that $P(m)$ is not identically zero. Indeed,
\begin{equation} \label{eq:PPneval}
P\left(\frac{n}{2}\right)=a_n 2^n\, 4 p_{n+1}\left(\frac{n}{2}\right) = a_n 2^{n+2}\, S(n+1, n+1)\left(\frac{-n}{2}\right)_{n+1}\left(\frac{1}{2}\right)_{n+1},
\end{equation}
which is non-zero due to the fact that we have assumed $n$ is odd, $a_n\neq 0$ and, as is well-known, Stirling numbers of the second kind satisfy $S(n+1,n+1)=1$.
\end{proof}

It is worthwhile to note that equation (\ref{s2m4}) is only valid for $m>n$ and our choice of $m=n/2$ in the preceding proof does not satisfy this. However, this fact is irrelevant to the argument. All that is required is that the polynomial $P(m)$ not be identically zero and this can be achieved by considering any values for $m$ that we choose.

Next, we want to determine conditions under which an even degree polynomial may interpolate a multiplier sequence for the Legendre basis. 
\begin{thm}\label{x^2+x} If $\{ \gamma_k \}_{k=0}^{\infty}$ is a multiplier sequence for the Legendre basis which is interpolated by a polynomial $p$, then $p(x)=h(x^2+x)$ for some polynomial $h(x)\in \mathbb{R}[x]$.
\end{thm}
\begin{proof} By Theorem \ref{thm:noodd}, $p$ must have even degree. Noting that 
\begin{equation}
b_k(x) = 
\begin{cases}
(x^2+x)^{k/2}  & \qquad k \text{ even},\\
x^k & \qquad k \text{ odd}.\\
\end{cases}
\end{equation}
forms a basis for $\mathbb{R}[x]$, we can expand $p$ in terms of this basis to obtain
$$
p(x) = h(x^2+x) + q(x),
$$
where $\deg h = (\deg p)/2$ and $q$ is an odd function in $\mathbb{R}[x]$. We claim that $q$ must be identically zero. By way of contradiction, suppose $q$ is not identically zero. Then $q$ is a polynomial of odd degree. Using equation (\ref{s2m1}) to calculate $S_{2m}(0)$, we have
\begin{align*}
S_{2m}(0) &=  \frac{1}{2\cdot 4^m m!} \sum_{k=0}^{m} \binom{m}{k}\frac{(4k + 1) p(2k) (-1)^{k}}{(k+1/2)_{m+1}}\\ 
          &=  \frac{1}{2\cdot 4^m m!} \sum_{k=0}^{m} \binom{m}{k}\frac{(4k + 1) \left[h((2k)^2+(2k))+q(2k)\right] (-1)^{k}}{(k+1/2)_{m+1}}.
\end{align*}
Splitting up the sum then gives
\begin{align}
S_{2m}(0)	&=  \frac{1}{2\cdot 4^m m!} \sum_{k=0}^{m} \binom{m}{k}\frac{(4k + 1) h((2k)^2+(2k)) (-1)^{k}}{(k+1/2)_{m+1}} \label{sumdecomp1}\\
					&\hskip .2 in  + \frac{1}{2\cdot 4^m m!} \sum_{k=0}^{m} \binom{m}{k}\frac{(4k + 1) q(2k) (-1)^{k}}{(k+1/2)_{m+1}} \label{sumdecomp2}.
\end{align}

The sum in (\ref{sumdecomp1}) calculates the polynomial coefficients, evaluated at zero, of the differential operator $T_h$ corresponding to the sequence $\{h(k^2+k) \}_{k=0}^{\infty}$. Since $T_h$ is a finite order differential operator, this sum must vanish for all sufficiently large $m$. Thus, for all such $m$, only the sum (\ref{sumdecomp2}) contributes to $S_{2m}(0)$, i.e.,
$$
S_{2m}(0) = \frac{1}{2\cdot 4^m m!} \sum_{k=0}^{m} \binom{m}{k}\frac{(4k + 1) q(2k) (-1)^{k}}{(k+1/2)_{m+1}}.
$$
Note that these are precisely the values of the coefficient polynomials, evaluated at zero, corresponding to the odd degree interpolated sequence $\seq{q(k)}$.  The result now follows from the proof of Theorem \ref{thm:noodd}.
\end{proof}

As an immediate consequence, we get a new proof of the following result which was proved in \cite{fhms}.
\begin{cor}
If $\{k^2+bk+c\}$ is a Legendre MS, then $b=1$.
\end{cor}

Theorem \ref{x^2+x} also allows us to conclude that, as a differential operator, any polynomially interpolated multiplier sequence for the Legendre basis must have finite order. 

\begin{cor}
Suppose $\{p(k)\}_{k=0}^{\infty}$ is a multiplier sequence for the Legendre basis, where $p\in\mathbb{R}[x]$, and let $T$ be the corresponding linear operator defined by $T[P_n(x)] = p(n)P_n(x)$ for all $n$. Then the differential operator representation of $T$ has only finitely many terms:
$$
T = \sum_{k=0}^{m} S_k(x) D^k.
$$
\end{cor}

\begin{proof}
Let $h$ be a real polynomial for which $p(x) = h(x^2+x)$ and write
$$
h(x) = a_0+ a_1 x + \cdots +a_n x^n.
$$
Let $\delta$ be the differential operator from Legendre's differential equation
$$
\delta = (x^2-1)D^2+ 2 x D.
$$
Then $\delta P_k(x) = (k^2+k)P_k(x)$ for all $k$ and it follows that the operator $T$ can be written as
$$
T = h(\delta) = a_0 + a_1 \delta + \cdots a_n \delta^n. 
$$
From this we see that the order of $T$ is at most $2n=\deg p$.
\end{proof}
\begin{rem}
The previous corollary complements a result of Miranian \cite{Miranian}, which implies that a multiplier sequence for the Legendre basis whose corresponding differential operator has \emph{finite} order must have the form $\{h(k^2+k)\}_{k=0}^{\infty}$ where $h\in\mathbb{R}[x].$ In fact, we have shown that any multiplier sequence for the Legendre basis whose corresponding differential operator has infinite order (of which, to date, none have been discovered) cannot be interpolated by a polynomial. 
\end{rem}

\section{A Collection of Legendre Multiplier Sequences}\label{s:collection}

In this section, we demonstrate that a certain collection of sequences are multiplier sequences for the Legendre basis. The result is reminiscent of a similar result by Craven and Csordas involving the characterization of polynomially interpolated classical complex zero decreasing sequences (see \cite[Prop. 2.2]{CCczds}). We will use the Bates-Yoshida Quadratic Hyperbolicity Preserver Characterization:

\begin{thm}[Bates-Yoshida \cite{BY}] \label{BY-thm}
Suppose $Q_2,Q_1,Q_0$ are real polynomials such that $deg(Q_2)=2$, $deg(Q_1)\le 1$, $deg(Q_0)=0$. Then
\begin{equation} \nonumber
T=Q_2D^2+Q_1D+Q_0
\end{equation}
preserves hyperbolicity if and only if
\begin{equation} \nonumber
W[Q_0,Q_2]^2-W[Q_0,Q_1]W[Q_1,Q_2]\le 0, \;\; \mbox{and} \;\; Q_0 \ll Q_1 \ll Q_2.
\end{equation}
\end{thm}

The compact form of Bates-Yoshida Theorem is derived using the Borcea-Br\"and\'en Characterization of Hyperbolicity Preservers \cite{bb}[Theorem 5], and intricate arguments involving the location of zeros of the coefficient polynomials. With Theorem \ref{BY-thm} in hand, we prove the following.

\begin{thm}\label{fall}
Let $\delta = (x^2-1)D^2+2 x D$ where $D=d/dx$ and fix positive integers $n$ and $N$. The operator
\begin{equation}
\label{fallfac}
\delta (\delta-1\cdot 2)(\delta-2\cdot 3)\cdots\big(\delta-(n-1)(n)\big)\prod_{j=1}^{N} (\delta-A_j)
\end{equation}
is a hyperbolicity preserver whenever $-(n+1)\leq A_j \leq n(n+1)$ for all $j$. 
\end{thm}

\begin{proof}
As in the proof of Theorem 14 on p. 137 of \cite{nreup}, we may factor the operator
$$
\delta (\delta-1\cdot 2)(\delta-2\cdot 3 )\cdots(\delta-(n-1)\cdot n)
$$
into a product of operators:
$$
\left(\prod_{k=0}^{n-1}\left[(x^2-1)D+2(k+1)x\right]\right)D^{n}.
$$
From Leibniz' rule (or by using equations (6) and (7) on p. 135 of \cite{nreup}), we have
$$
D^{n} (\delta-A_j) = \left[(x^2-1)D^2 + 2(n+1) x D + n^2+n-A_j\right] D^{n}.
$$
It follows that the operator in (\ref{fallfac}) can be factored as
\begin{equation}\label{ops}
\left(\prod_{k=0}^{n-1}\left[(x^2-1)D+2(1+k)x\right]\right) T D^{n},
\end{equation}
where
\begin{equation}\label{Top}
T = \left(\prod_{j=1}^{N} \left[(x^2-1)D^2 + 2(n+1) x D + n^2+n-A_j\right] \right).
\end{equation}
By Theorem 8 of \cite{nreup}, any operator of the form 
$$
q(x) D + \alpha q'(x), 
$$
where $q\in\mathbb{R}[x]$ has only real zeros and $\alpha\geq 0$, is a complex zero decreasing (and, therefore, hyperbolicity preserving) operator. Thus, the operators in the product, along with $D^n$, of equation (\ref{ops}) are hyperbolicity operators. It remains to show that the operators appearing in the product for $T$ in equation (\ref{Top}) are hyperbolicity preserving. We do this by applying Theorem \ref{BY-thm}. It is clear that the coefficient polynomials of the operator are in proper position. Thus, we only need to find conditions on $A_j$ under which
$$
W[n^2+n-A_j, x^2-1]^2 - W[n^2+n-A_j, 2(n+1)x]W[2(n+1)x, x^2-1]\leq 0
$$
for all $x\in\R$, where $W[f,g]$ denotes the Wronskian of $f$ and $g$
$$
W[f,g] = fg'-f'g. 
$$
A calculation shows that the inequality in question reduces to 
$$
-4(n^2+n-A_j)\left[(A_j+n+1)x^2+4(n+1)^2\right]\leq 0,
$$
and the result follows.

\end{proof}

The operator $\delta$ appearing in Theorem \ref{fall} satisfies $\delta P_k(x) = (k^2+k)P_k(x)$. From this, we obtain an immediate corollary.

\begin{cor}\label{fallseq}
Fix positive integers $n$ and $N$, and let 
\begin{equation}
h(x) = x (x-1\cdot 2)(x-2\cdot 3)\cdots\big(x-(n-1)(n)\big)\prod_{j=1}^{N} (x-A_j),
\end{equation}
where $-(n+1)\leq A_j \leq n(n+1)$ for all $j$. Then $\{h(k^2+k)\}_{k=0}^{\infty}$ is a multiplier sequence for the Legendre basis. 
\end{cor}

We suspect that these results are not as sharp as they could be. For example, Theorem \ref{fall}, indicates that operators of the form $\delta(\delta+A)$ are hyperbolicity preserving for $-2\leq A\leq 2$ but we suspect that these are hyperbolicity preserving for $-2\leq A\leq 4\sqrt{2}-2$ (we elucidate on this point in section \ref{s:symbolcurve} below). Furthermore, based on the similarity of the CZDS results, we suspect the following problem can be answered in the affirmative.

\begin{problem}
Is it true that an operator of the form $(\ref{fallfac})$ is a complex zero decreasing operator if and only if $-n(n+1)\leq A_j \leq n+1$ for all $j\in\{1, 2, 3, \dots, N\}$?
\end{problem}

\section{Geometry of the Symbol Curve}\label{s:symbolcurve}

In this final section, we use a beautiful result contained in \cite{BBmv} to pose a conjecture about the classification of quartic polynomials which interpolate multiplier sequences for the Legendre basis. The result we refer to deserves to be better known and we state a special case of it here for the convenience of the reader. 
\begin{thm}\label{curve}\emph{(see \cite[Corollary 4.4.5]{BBmv})} Let $T$ be a finite order differential operator with
$$
T = \sum_{k=0}^{n} Q_k(x) D^k \qquad \left(D=\frac{d}{dx}\right).
$$
Then $T$ is hyperbolicity preserving if and only if the symbol curve in $\mathbb{R}^2$
$$
0 = \sum_{k=0}^{n} (-1)^k Q_k(x) y^k 
$$
has $n$ intersections (counted with multiplicity) with every line of positive slope. 
\end{thm}
Note that the symbol curve can be calculated as $T[\exp(-xy)]$, where $T$ acts on the variable $x$ alone.

Now, by the results of Section \ref{s:Form and Order}, any quartic Legendre MS (up to a multiplicative constant) has the form
$$
\{(k^2+k)^2+ b (k^2+k) + c\}_{k=0}^{\infty}.
$$

We will first focus on the case $c=0$. As noted after Corollary \ref{fallseq}, if $-2\leq b \leq 2$, the sequence is a multiplier sequence for the Legendre basis. 

Now, since every Legendre MS is also a classical multiplier sequence, we can rule out several values of the parameter $b$. The sequence begins
$$
0, 4+2b, \dots
$$
and, since the sequence will eventually always be positive, we must have that all the terms are nonnegative. Therefore, any such sequence with $b<-2$ cannot be a (classical and, therefore,) multiplier sequence for the Legendre basis. Furthermore, applying the sequence to $e^x$ in the standard basis yields
$$
T[e^x] = \frac{(k^2+k)^2+b(k^2+k)}{k!} x^k = (x^2+6x+2+b)(2+x)x e^x
$$
with zeros
$$
x=0, x=-2, x= -3 \pm \sqrt{7-b}
$$
from which we see that any such sequence with $b>7$ cannot be a (classical and, therefore,) multiplier sequence for the Legendre basis. It now remains to find out what happens for $2<b\leq 7$. Using Theorem \ref{curve}, we want to determine conditions under which every line with positive slope will have the correct number of intersections with the curve
$$
T[e^{-xw}]=0.
$$
For our quartic operator, this is equivalent to examining the intersection property for the curve  
\begin{equation}\label{symbolcurve}
14 x^2 w^2-8 x^3 w^3+x^4 w^4-2 w^4 x^2-6 w^2+8 x w^3+w^4-4 x w+b x^2 w^2-b w^2-2 b x w+c = 0.
\end{equation}
In Figure \ref{symbol3}, we have graphed this curve for $c=0$ and $b=3$, $b=3.6569$, and $b=4$. 
\vskip .1 in 
\begin{figure}[h]
\centering
\includegraphics[width=1.5in]{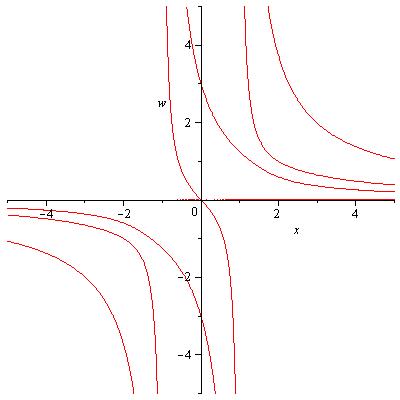}
\hskip .2 in
\includegraphics[width=1.5in]{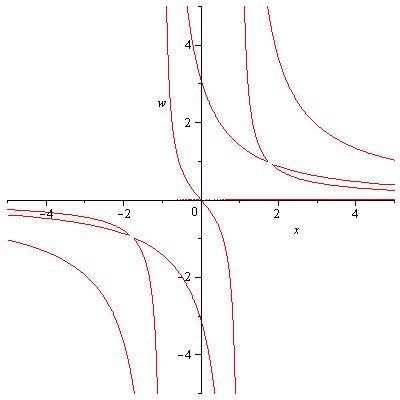}
\hskip .2 in 
\includegraphics[width=1.5in]{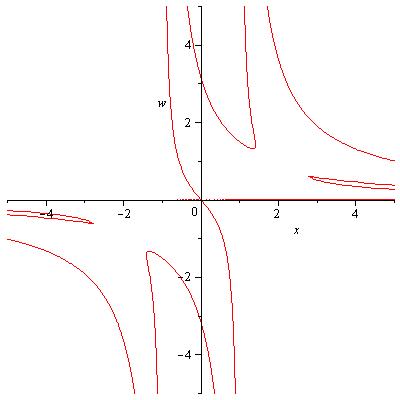}
\caption{The symbol curve for $c=0$ and various choices of $b$.}
\label{symbol3}
\end{figure}
We can see a ``breaking point'' somewhere around $b=3.6569$. To find the exact value of $b$, we seek to find when (\ref{symbolcurve}) with $c=0$ has multiple zeros. An analysis of the discriminant then shows that the target value is $b=4\sqrt{2}-2\approx 3.656854248$. 

For general values of $c$, a similar analysis has given us reason to believe that the following conjecture is true.

\begin{conjecture}
For the sequence $\{(k^2+k)^2+ b (k^2+k)+c\}_{k=0}^{\infty}$ to be a multiplier sequence for the Legendre basis, it is necessary and sufficient that $b$ and $c$ lie in the region in the $bc$-plane bounded by the $b$-axis, the parabola $c=(b+2)^2/8-4$, and the curve
\begin{align*}
0&=b^5- (c+15) b^4+ 4(c+12)b^3+8 (c^2+24 c + 19) b^2- 16(10 c^2+101 c+ 33)b\\
 &\hskip .2 in - 16( c^3 -50 c^2 - 185 c +63),
\end{align*}
as depicted in Figure \ref{sharkfin}.
\end{conjecture}
\begin{figure}[h]
\centering
\includegraphics[width=1.5 in]{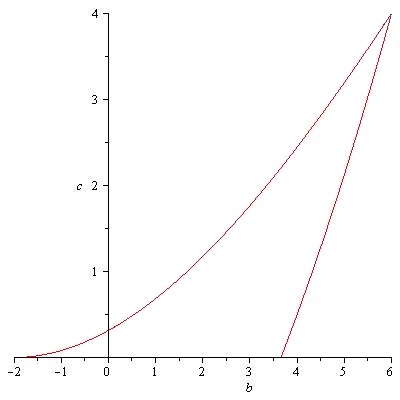}
\caption{Conjectured region of allowable values for the parameters $b$ and $c$.}
\label{sharkfin}
\end{figure} 
\vskip .2 in 

Similarly, using the symbol curve as our guide, we believe the following conjecture to be true as well.
\begin{conjecture}
Let $n$ be a positive integer and suppose $k$ is a positive integer that is at most $n-1$. Then the operator 
\begin{equation}\label{2^k}
\delta^{n-k}(\delta^k - 2^k) \qquad \qquad (\delta = (x^2-1)D^2+ 2 x D)
\end{equation}
is a hyperbolicity preserving operator.
\end{conjecture}

We note that many cases of the previous conjecture can be verified using Theorem \ref{fall}. However, there are operators of the form (\ref{2^k}) that remain mysterious, such as $\delta(\delta^3-8)$. 
 
Finally, we leave the reader to consider the central problem which remains open for the time being.
\begin{problem}
Characterize all multiplier sequences for the Legendre basis.
\end{problem}

\end{document}